\documentclass[11pt]{article}
\usepackage{amsfonts}
\usepackage{amsmath}
\usepackage{amssymb}
\usepackage{amsthm}
\usepackage{latexsym}
\usepackage{mathrsfs}
\usepackage{mathtools}

\newtheorem{theorem}{\bf Theorem}[section]
\newtheorem{lemma}[theorem]{\bf Lemma}

\newtheorem{remark}[theorem]{\bf Remark}

\numberwithin{equation}{section}

\begin{document}
%------------------------------------------------------------------------------
\vspace*{0ex}
\begin{center}
{\Large\bf
A Hamiltonian structure of the Isobe--Kakinuma model \\[0.5ex] for water waves}
\end{center}

\begin{center}
{\em Dedicated to the late Professor Walter L. Craig}
\end{center}

\begin{center}
Vincent Duch\^ene and Tatsuo Iguchi
\end{center}

\begin{abstract}
We consider the Isobe--Kakinuma model for water waves, which is obtained as the system 
of Euler--Lagrange equations for a Lagrangian approximating Luke's Lagrangian for water waves. 
We show that the Isobe--Kakinuma model also enjoys a Hamiltonian structure analogous to
the one exhibited by V. E. Zakharov on the full water wave problem 
and, moreover,  that the Hamiltonian of the Isobe--Kakinuma model is a higher order 
shallow water approximation to the one of the full water wave problem. 
\end{abstract}

%------------------------------------------------------------------------------
\section{Introduction}
\label{section:intro}
We consider a model for the motion of water in a moving domain of the $(n+1)$-dimensional 
Euclidean space. 
The water wave problem is mathematically formulated as a free boundary problem for 
an irrotational flow of an inviscid, incompressible, and homogeneous fluid under a vertical gravitational field. 
Let $t$ be the time, $\boldsymbol{x}=(x_1,\ldots,x_n)$ the horizontal spatial coordinates, 
and $z$ the vertical spatial coordinate.
We assume that the water surface and the bottom are represented as $z=\eta({\boldsymbol x},t)$ and 
$z=-h+b({\boldsymbol x})$, respectively, where $\eta({\boldsymbol x},t)$ is the surface elevation, 
$h$ is the mean depth, and $b({\boldsymbol x})$ represents the bottom topography. 
We denote by $\Omega(t)$, $\Gamma(t)$, and $\Sigma$ the water region, the water surface, and the 
bottom of the water at time $t$, respectively. 
Then, the motion of the water is described by the velocity potential $\Phi({\boldsymbol x},z,t)$ satisfying 
Laplace's equation 
\begin{equation}\label{intro:Laplace}
\Delta\Phi + \partial_z^2\Phi = 0 \makebox[3em]{in} \Omega(t), \; t>0, 
\end{equation}
where $\Delta=\partial_{x_1}^2+\cdots+\partial_{x_n}^2$. 
The boundary conditions on the water surface are given by 
\begin{equation}\label{intro:BC1}
\begin{cases}
 \partial_t\eta + \nabla\Phi\cdot\nabla\eta - \partial_z\Phi = 0 & \mbox{on}\quad \Gamma(t), \; t>0, \\
 \displaystyle
 \partial_t\Phi +\frac12 \bigl( |\nabla\Phi|^2 + (\partial_z\Phi)^2 \bigr) + g\eta = 0
  & \mbox{on}\quad \Gamma(t), \; t>0,
\end{cases}
\end{equation}
where $\nabla=(\partial_{x_1},\ldots,\partial_{x_n})^{\rm T}$, and $g$ is the gravitational constant. 
The first equation is the kinematic condition on the water surface and the second one is Bernoulli's equation. 
Finally, the boundary condition on the bottom of the water is given by 
\begin{equation}\label{intro:BC2}
\nabla\Phi\cdot\nabla b - \partial_z\Phi = 0 \makebox[3em]{on} \Sigma, \; t>0,
\end{equation}
which is the kinematic condition on the fixed and impermable bottom.
These are the basic equations for the water wave problem.

We put 
\begin{equation}\label{intro:CV}
\phi({\boldsymbol x},t) = \Phi({\boldsymbol x},\eta({\boldsymbol x},t),t),
\end{equation}
which is the trace of the velocity potential on the water surface. 
Then, the basic equations for water waves~\eqref{intro:Laplace}--\eqref{intro:BC2} are transformed 
equivalently into 
\begin{equation}\label{intro:ZCS}
\begin{cases}
 \partial_t\eta-\Lambda(\eta,b)\phi = 0  & \mbox{on}\quad \mathbf{R}^n, \; t>0, \\[.5ex]
 \partial_t\phi + g\eta + \dfrac12|\nabla\phi|^2
   - \dfrac12\dfrac{\bigl(\Lambda(\eta,b)\phi + \nabla\eta\cdot\nabla\phi\bigr)^2}{1+|\nabla\eta|^2} = 0
    & \mbox{on}\quad \mathbf{R}^n, \; t>0,
\end{cases}
\end{equation}
where $\Lambda(\eta,b)$ is the Dirichlet-to-Neumann map for Laplace's equation. 
Namely, it is defined by
\[
\Lambda(\eta,b)\phi= (\partial_z\Phi)|_{z=\eta} - \nabla\eta\cdot(\nabla\Phi)|_{z=\eta},
\]
where $\Phi$ is the unique solution to the boundary value problem of Laplace's equation 
\eqref{intro:Laplace} under the boundary conditions \eqref{intro:BC2}--\eqref{intro:CV}.

As is well-known, the water wave problem has a conserved energy $E=E_{\rm kin} + E_{\rm pot}$, 
where $E_{\rm kin}$ is the kinetic energy 
\begin{align*}
E_{\rm kin} 
&= \frac{1}{2}\rho\iint_{\Omega(t)} \bigl( |\nabla\Phi({\boldsymbol x},z,t)|^2 
 +  (\partial_z\Phi({\boldsymbol x},z,t))^2  \bigr)\,{\rm d}{\boldsymbol x}{\rm d}z \\
&= \frac{1}{2}\rho\int_{\mathbf{R}^n}\phi({\boldsymbol x},t)(\Lambda(\eta,b)\phi)({\boldsymbol x},t)\,{\rm d}{\boldsymbol x},
\end{align*}
and $E_{\rm pot}$ is the potential energy 
\[
E_{\rm pot} = \frac{1}{2}\rho g\int_{\mathbf{R}^n}\eta({\boldsymbol x},t)^2\,{\rm d}{\boldsymbol x}
\]
due to the gravity. Here, $\rho$ is a constant density of the water.

V. E. Zakharov~\cite{Zakharov1968} found that the water wave system has a Hamiltonian structure and 
$\eta$ and $\phi$ are the canonical variables. 
The Hamiltonian $\mathscr{H}$ is essentially the total energy, that is, $\mathscr{H} = \frac{1}{\rho}E$. 
He showed that the basic equations for water waves \eqref{intro:Laplace}--\eqref{intro:BC2} are transformed 
equivalently into Hamilton's canonical equations
\[
\partial_t\eta = \frac{\delta\mathscr{H}}{\delta\phi}, \quad
\partial_t\phi = -\frac{\delta\mathscr{H}}{\delta\eta}.
\]
Although V. E. Zakharov did not use explicitly the Dirichlet-to-Neumann map $\Lambda(\eta,b)$, 
the above canonical equations are exactly the same as \eqref{intro:ZCS}. 
W. Craig and C. Sulem~\cite{CraigSulem1993} introduced the Dirichlet-to-Neumann map explicitly 
and derived \eqref{intro:ZCS}. 
Therefore, nowadays \eqref{intro:ZCS} is often called the Zakharov--Craig--Sulem 
formulation of the water wave problem. 
Since then, W. Craig and his collaborators \cite{CraigGroves1994,CraigGroves2000, CraigGuyenneKalisch2005, 
CraigGuyenneNichollsSulem2005, CraigGuyenneSulem2010, CraigGuyenneSulem2012} have used the Hamiltonian 
structure of the water wave problem in order to analyze long-wave and modulation approximations. 
Let us also mention the recent work of W. Craig~\cite{Craig2016}, which generalizes the Hamiltonian 
formulation of water waves described above to a general coordinatization of the free surface 
allowing overturning wave profiles.

On the other hand, as was shown by J. C. Luke~\cite{Luke1967}, 
the water wave problem has also a variational structure. 
His Lagrangian density is of the form 
\begin{equation}\label{intro:Luke's Lagrangian}
\mathscr{L}(\Phi,\eta) = \int_{-h+b({\boldsymbol x})}^{\eta({\boldsymbol x},t)}\biggl(
 \partial_t\Phi({\boldsymbol x},z,t) + \frac12\bigl( |\nabla\Phi({\boldsymbol x},z,t)|^2
  +  (\partial_z\Phi({\boldsymbol x},z,t))^2 \bigr) \biggr)\,{\rm d}z
 + \frac12g\bigl(\eta(\boldsymbol{x},t)\bigr)^2 
 \end{equation}
and the action function is given by 
\[
\mathscr{J}(\Phi,\eta)
 = \int_{t_0}^{t_1}\!\!\!\int_{\mathbf{R}^n}\mathscr{L}(\Phi,\eta)\,{\rm d}{\boldsymbol x}\,{\rm d}t.
\]
In fact, the corresponding Euler--Lagrange equations are exactly the basic equations for water waves 
\eqref{intro:Laplace}--\eqref{intro:BC2}. 
We refer to J. W. Miles~\cite{Miles1977} for the relation between Zakharov's Hamiltonian and Luke's Lagrangian.

M. Isobe~\cite{Isobe1994, Isobe1994-2} and T. Kakinuma~\cite{Kakinuma2000, Kakinuma2001, Kakinuma2003} 
obtained a family of systems of equations after replacing the velocity potential $\Phi$ in Luke's Lagrangian by 
\[
\Phi^{\rm app}({\boldsymbol x},z,t) = \sum_{i=0}^N\Psi_i(z;b)\phi_i({\boldsymbol x},t),
\]
where $\{\Psi_i\}$ is a given appropriate function system in the vertical coordinate $z$ and may depend on 
the bottom topography $b$ and $(\phi_0,\phi_1,\ldots,\phi_N)$ are unknown variables. 
The Isobe--Kakinuma model is a system of Euler--Lagrange equations corresponding to the action function 
\begin{equation}\label{intro:IK-action}
\mathscr{J}^{\rm app}(\phi_0,\phi_1,\ldots,\phi_N,\eta)
 = \int_{t_0}^{t_1}\!\!\!\int_{\mathbf{R}^n}\mathscr{L}(\Phi^{\rm app},\eta)
  \,{\rm d}{\boldsymbol x}{\rm d}t.
\end{equation}
We have to choose the function system $\{\Psi_i\}$ carefully for the Isobe--Kakinuma model 
to produce good approximations to the water wave problem. 
One possible choice is the bases of the Taylor series of the velocity potential 
$\Phi({\boldsymbol x},z,t)$ with respect to the vertical spatial coordinate $z$ around the bottom. 
Such an expansion has been already used by J. Boussinesq~\cite{Boussinesq1872} in the case of the flat bottom 
and, for instance, by C. C. Mei and B. Le M\'ehaut\'e~\cite{MeiLeMehaute1966} for general bottom topographies. 
The corresponding choice of the function system is given by 
\[
\Psi_i(z;b) 
 = 
\begin{cases}
 (z + h)^{2i} & \mbox{in the case of the flat bottom}, \\
 (z + h - b({\boldsymbol x}))^i & \mbox{in the case of the variable bottom}.
 \end{cases}
\]
Here we note that the latter choice is valid also for the case of the flat bottom. 
However, it turns out that the terms of odd degree do not play any important role in such a case 
so that the former choice is more adequate. 
In order to treat both cases at the same time, we adopt the approximation 
\begin{equation}\label{intro:app}
\Phi^{\rm app}({\boldsymbol x},z,t)
 = \sum_{i=0}^N(z+h-b({\boldsymbol x}))^{p_i}\phi_i({\boldsymbol x},t),
\end{equation}
where $p_0,p_1,\ldots,p_N$ are nonnegative integers satisfying $0=p_0<p_1<\cdots<p_N$. 
Plugging~\eqref{intro:app} into the action function~\eqref{intro:IK-action}, 
the corresponding Euler--Lagrange equation yields the Isobe--Kakinuma model of the form 
\begin{equation}\label{intro:IK model}
\left\{
 \begin{array}{l}
  \displaystyle
  H^{p_i} \partial_t \eta + \sum_{j=0}^N \left\{ \nabla \cdot \left(
   \frac{1}{p_i+p_j+1} H^{p_i+p_j+1} \nabla\phi_j
   - \frac{p_j}{p_i+p_j} H^{p_i+p_j} \phi_j \nabla b \right) \right.\\
  \displaystyle\phantom{ H^{p_i} \partial_t \eta + \sum_{j=0}^N \biggl\{ }
   + \left. \frac{p_i}{p_i+p_j} H^{p_i+p_j} \nabla b \cdot \nabla\phi_j
   - \frac{p_ip_j}{p_i+p_j-1} H^{p_i+p_j-1} (1 + |\nabla b|^2) \phi_j \right\} = 0 \\
  \makebox[28em]{}\mbox{for}\quad i=0,1,\ldots,N, \\[1ex]
  \displaystyle
  \sum_{j=0}^N H^{p_j} \partial_t \phi_j + g\eta 
   + \frac12 \left\{ \left| \sum_{j=0}^N ( H^{p_j}\nabla\phi_j - p_j H^{p_j-1}\phi_j\nabla b ) \right|^2 
   + \left( \sum_{j=0}^N p_j H^{p_j-1} \phi_j \right)^2 \right\} = 0,
 \end{array}
\right.
\end{equation}
where $H({\boldsymbol x},t) = h + \eta({\boldsymbol x},t) - b({\boldsymbol x})$ is the depth of the water. 
Here and in what follows we use the notational convention $0/0=0$. 
This system consists of $(N+1)$ evolution equations for $\eta$ and only one evolution equation 
for $(N+1)$ unknowns $(\phi_0,\phi_1,\ldots,\phi_N)$, so that this is an overdetermined and 
underdetermined composite system. 
However, the total number of the unknowns is equal to the total number of the equations.

The main purpose of this paper is to show that the Isobe--Kakinuma model~\eqref{intro:IK model} also 
enjoys a canonical Hamiltonian structure which is analogous to the one of the water waves problem. 
In particular, the Hamiltonian is a higher order shallow water approximation of the original Hamiltonian 
of the water waves problem.

\medskip
\noindent
{\bf Acknowledgement} \\
T. I. was partially supported by JSPS KAKENHI Grant Number JP17K18742 and JP17H02856.

%------------------------------------------------------------------------------
\section{Preliminaries}
\label{section:IK}
Since the hypersurface $t=0$ in the space-time $\mathbf{R}^n\times\mathbf{R}$ 
is characteristic for the Isobe--Kakinuma model~\eqref{intro:IK model}, 
the initial value problem to the model is not solvable in general. 
In fact, if the problem has a solution $(\eta,\phi_0,\ldots,\phi_N)$, then by eliminating the time derivative 
$\partial_t\eta$ from the equations we see that the solution has to satisfy the relations 
\begin{align}\label{IK:comp}
& H^{p_i}\sum_{j=0}^N\nabla\cdot\left(
   \frac{1}{p_j+1}H^{p_j+1}\nabla\phi_j
   -H^{p_j}\phi_j\nabla b\right) \nonumber \\
&= \sum_{j=0}^N\left\{\nabla\cdot\left(
   \frac{1}{p_i+p_j+1}H^{p_i+p_j+1}\nabla\phi_j
   -\frac{p_j}{p_i+p_j}H^{p_i+p_j}\phi_j\nabla b\right)\right.  \\
&\phantom{ =\sum_{j=0}^N\biggl\{ }
 \displaystyle\left.
   +\frac{p_i}{p_i+p_j}H^{p_i+p_j}\nabla b\cdot\nabla\phi_j
   -\frac{p_ip_j}{p_i+p_j-1}H^{p_i+p_j-1}(1+|\nabla b|^2)\phi_j\right\} \nonumber 
\end{align}
for $i=1,\ldots,N$. 
Therefore, the initial data have to satisfy these relations in order to allow the existence of a solution. 
Y. Murakami and T. Iguchi~\cite{MurakamiIguchi2015} and R. Nemoto and T. Iguchi~\cite{NemotoIguchi2018} 
showed that the initial value problem to the Isobe--Kakinuma model~\eqref{intro:IK model} is well-posed 
locally in time in a class of initial data for which the relations~\eqref{IK:comp} and a generalized 
Rayleigh--Taylor sign condition are satisfied. 
Moreover, T. Iguchi~\cite{Iguchi2018-1, Iguchi2018-2} showed that the Isobe--Kakinuma model~\eqref{intro:IK model}
is a higher order shallow water approximation for the water wave problem 
in the strongly nonlinear regime. 
The Isobe--Kakinuma model~\eqref{intro:IK model} has also a conserved energy, which is the total 
energy given by 
\begin{align}
E^{\mbox{\rm\scriptsize IK}}(\eta,\boldsymbol{\phi})
&= \frac12\rho\iint_{\Omega(t)}\bigl( |\nabla\Phi^{\rm app}({\boldsymbol x},z,t)|^2
 +  (\partial_z\Phi^{\rm app}({\boldsymbol x},z,t))^2 \bigr)\,
 {\rm d}{\boldsymbol x}{\rm d}z
 + \frac12\rho g\int_{\mathbf{R}^n}\eta({\boldsymbol x},t)^2\,{\rm d}{\boldsymbol x} \nonumber \\
&= \frac{\rho}{2}\int_{\mathbf{R}^n}\left\{
 \sum_{i,j=0}^N\left(
  \frac{1}{p_i+p_j+1}H^{p_i+p_j+1}\nabla\phi_i\cdot\nabla\phi_j
  -\frac{2p_i}{p_i+p_j}H^{p_i+p_j}\phi_i\nabla b\cdot\nabla\phi_j \right. \right. \nonumber\\
& \left. \phantom{\sum_{i,j=0}^N}\makebox[6em]{}
 + \frac{p_ip_j}{p_i+p_j-1}H^{p_i+p_j-1}(1+|\nabla b|^2)\phi_i\phi_j\biggr) + g\eta^2 \right\}
 {\rm d}{\boldsymbol x},
 \label{IK:energy}
\end{align}
where $\boldsymbol{\phi}=(\phi_0,\phi_1,\ldots,\phi_N)^{\rm T}$.

We introduce second order differential operators $L_{ij}=L_{ij}(H,b)$ $(i,j=0,1,\ldots,N)$ depending on 
the water depth $H$ and the bottom topography $b$ by 
\begin{align}
L_{ij}\psi_j
&= -\nabla\cdot\biggl(
   \frac{1}{p_i+p_j+1}H^{p_i+p_j+1}\nabla\psi_j
   -\frac{p_j}{p_i+p_j}H^{p_i+p_j}\psi_j\nabla b\biggr) \nonumber \\[0.5ex]
&\quad\,
  -\frac{p_i}{p_i+p_j}H^{p_i+p_j}\nabla b\cdot\nabla\psi_j
   +\frac{p_ip_j}{p_i+p_j-1}H^{p_i+p_j-1}(1+|\nabla b|^2)\psi_j.
  \label{IK:L}
\end{align}
Then, we have $L_{ij}^*=L_{ji}$, where $L_{ij}^*$ is the adjoint operator of $L_{ij}$ in $L^2(\mathbf{R}^n)$. 
Moreover, the Isobe--Kakinuma model~\eqref{intro:IK model} and the relations~\eqref{IK:comp} can be written 
simply as 
\begin{equation}\label{IK:IK model}
\left\{
 \begin{array}{l}
  \displaystyle
  H^{p_i} \partial_t \eta - \sum_{j=0}^N L_{ij}(H,b) \phi_j = 0 \qquad\mbox{for}\quad i=0,1,\ldots,N, \\
  \displaystyle
  \sum_{j=0}^N H^{p_j} \partial_t \phi_j + g\eta
   + \frac12\bigl( |(\nabla\Phi^{\rm app})|_{z=\eta}|^2
    +  ((\partial_z\Phi^{\rm app})|_{z=\eta})^2  \bigr)  = 0
 \end{array}
\right.
\end{equation}
and 
\begin{equation}\label{IK:comp2}
\sum_{j=0}^N ( L_{ij}(H,b) - H^{p_i} L_{0j}(H,b) ) \phi_j = 0 \qquad\mbox{for}\quad i=1,\ldots,N,
\end{equation}
respectively. 
It is easy to calculate the variational derivative of the energy function 
$E^{\mbox{\rm\scriptsize IK}}(\eta,\boldsymbol{\phi})$ and to obtain 
\begin{equation}\label{IK:deltaE}
\begin{cases}
\displaystyle
\frac{1}{\rho} \delta_{\phi_i} E^{\mbox{\rm\scriptsize IK}}
 = \sum_{j=0}^N L_{ij}(H,b)\phi_j \quad j=0,1,\ldots,N, \\[3ex]
\displaystyle
\frac{1}{\rho} \delta_{\eta} E^{\mbox{\rm\scriptsize IK}}
 = \frac12\bigl(|(\nabla\Phi^{\rm app})|_{z=\eta}|^2
  + ((\partial_z\Phi^{\rm app})|_{z=\eta})^2  \bigr)  + g\eta.
\end{cases}
\end{equation}
Therefore, introducing $\boldsymbol{l}(H) = (H^{p_0},H^{p_1},\ldots,H^{p_N})^{\rm T}$, 
the Isobe--Kakinuma model~\eqref{intro:IK model} can also be written simply as 
\begin{equation}\label{IK:IK model2}
\begin{pmatrix}
 0 & -\boldsymbol{l}(H)^{\rm T} \\
 \boldsymbol{l}(H) & O
\end{pmatrix}
\partial_t
\begin{pmatrix}
 \eta \\
 \boldsymbol{\phi}
\end{pmatrix}
= \frac{1}{\rho}
\begin{pmatrix}
 \delta_{\eta} E^{\mbox{\rm\scriptsize IK}} \\
 \delta_{\boldsymbol{\phi}} E^{\mbox{\rm\scriptsize IK}}
\end{pmatrix}.
\end{equation}

In view of~\eqref{IK:comp2} we introduce also linear operators $\mathcal{L}_i = \mathcal{L}_i(H,b)$ 
$(i = 1,\ldots,N)$ depending on the water depth $H$ and the bottom topography $b$, 
and acting on $\mbox{\boldmath$\varphi$} = (\varphi_0,\varphi_1,\ldots,\varphi_N)^{\rm T}$ by 
\begin{equation}\label{IK:scrLi}
\mathcal{L}_i \mbox{\boldmath$\varphi$} = \sum_{j=0}^N \bigl( L_{ij}(H,b) - H^{p_i} L_{0j}(H,b) \bigr) \varphi_j
 \quad\mbox{for}\quad i=1,\ldots,N, 
\end{equation}
and put $\mathcal{L} \mbox{\boldmath$\varphi$}
 = (\mathcal{L}_1 \mbox{\boldmath$\varphi$},\ldots,\mathcal{L}_N \mbox{\boldmath$\varphi$})^{\rm T}$. 
Then, the conditions~\eqref{IK:comp} can be written simply as 
\begin{equation}\label{IK:comp3}
 \mathcal{L}(H,b)\boldsymbol{\phi} = {\bf 0} .
\end{equation}
For later use, we also put $L=L(H,b)=(L_{ij}(H,b))_{0\leq i,j\leq N}$ and define $L_0=L_0(H,b)$ by 
\begin{equation}\label{IK:L0}
L_0(H,b)\boldsymbol{\varphi} = \sum_{j=0}^NL_{0j}(H,b)\varphi_j.
\end{equation}
Then, the conditions~\eqref{IK:comp} are also equivalent to 
\begin{equation}\label{IK:comp4}
L(H,b)\boldsymbol{\phi} = \bigl(L_0(H,b)\boldsymbol{\phi}\bigr)\boldsymbol{l}(H).
\end{equation}
Now, for given functions $F_0$ and $\boldsymbol{F}=(F_1,\ldots,F_N)^{\rm T}$ we consider the equations 
\begin{equation}\label{IK:elliptic}
\begin{cases}
 \boldsymbol{l}(H)\cdot \boldsymbol{\varphi}  =F_0,\\
  \mathcal{L}(H,b)\boldsymbol{\varphi} = \boldsymbol{F}.
  \end{cases}
\end{equation}
Let $W^{m,p}=W^{m,p}(\mathbf{R}^n)$ be the $L^p$-based Sobolev space of order $m$ on $\mathbf{R}^n$ 
and $H^m=W^{m,2}$. 
The norms of the Sobolev space $H^m$ and of a Banach space $\mathscr{X}$ 
are denoted by $\|\cdot\|_m$ and by $\|\cdot\|_\mathscr{X}$, respectively. 
Set $\mathring{H}^m=\{\phi \,;\, \nabla\phi\in H^{m-1} \}$. 
The following lemma was proved in~\cite{NemotoIguchi2018}.

\begin{lemma}\label{IK:lem1}
Let $h,c_0,M$ be positive constants and $m$ an integer such that $m>\frac{n}{2}+1$. 
There exists a positive constant $C$ such that if $\eta$ and $b$ satisfy 
\begin{equation}\label{IK:cond}
\left\{
 \begin{array}{l}
  \|\eta\|_m+\|b\|_{W^{m,\infty}} \leq M, \\[0.5ex]
  c_0\leq H({\boldsymbol x})=h+\eta({\boldsymbol x})-b({\boldsymbol x}) \quad\mbox{for}\quad {\boldsymbol x}\in\mathbf{R}^n,
 \end{array}
\right.
\end{equation}
then for any $F_0\in \mathring{H}^k$ and $\mbox{\boldmath$F$}=(F_1,\ldots,F_N)^{\rm T}\in (H^{k-2})^N$ 
with $1\leq k\leq m$ there exists a unique solution 
$\mbox{\boldmath$\varphi$}=(\varphi_0,\varphi_1,\ldots,\varphi_N)^{\rm T}\in \mathring{H}^k\times  (H^{k})^N$ 
to~\eqref{IK:elliptic}. 
Moreover, the solution satisfies
\[
\|\nabla\varphi_0\|_{k-1}+\|(\varphi_1,\ldots,\varphi_N)\|_k
\leq C\;(\|\nabla F_0\|_{k-1}+\|(F_1,\ldots,F_N)\|_{k-2}).
\]
\end{lemma}

%------------------------------------------------------------------------------
\section{Hamiltonian structure}
\label{section:HS}
In the following, we will fix $b\in W^{m,\infty}$ with $m>\frac{n}{2}+1$. 
Let $(\eta,\phi_0,\ldots,\phi_N)$ be a solution to the Isobe--Kakinuma model~\eqref{intro:IK model}. 
As we will see later, the canonical variables of the Isobe--Kakinuma model are the surface elevation 
$\eta$ and the trace of the approximated velocity potential on the water surface 
\begin{equation}\label{H:CV}
\phi = \Phi^{\rm app}|_{z=\eta} = \sum_{j=0}^NH^{p_j}\phi_j
 = \boldsymbol{l}(H)\cdot\boldsymbol{\phi}.
\end{equation}
Then, the relations~\eqref{IK:comp} and the above equation are written in the simple form 
\begin{equation}\label{H:eqCV}
\begin{cases}
\boldsymbol{l}(H)\cdot \boldsymbol{\phi}  =\phi,\\
\mathcal{L}(H,b)\boldsymbol{\phi} = {\bf 0}.
  \end{cases}
\end{equation}
Therefore, it follows from Lemma~\ref{IK:lem1} that once the canonical variables 
$(\eta,\phi)$ are given in an appropriate class of functions, 
$\boldsymbol{\phi}=(\phi_0,\phi_1,\ldots,\phi_N)^{\rm T}$ can be determine uniquely. 
In other words, these variables $(\phi_0,\phi_1,\ldots,\phi_N)$ depend on the canonical variables 
$(\eta,\phi)$ and furthermore they depend on $\phi$ linearly so that we can write 
\[
\boldsymbol{\phi} = {\bf S}(\eta,b)\phi
\]
with a linear operator ${\bf S}(\eta,b)$ depending on $\eta$ and $b$. 
Since we fixed $b$, we simply write ${\bf S}(\eta)$ in place of ${\bf S}(\eta,b)$ for simplicity.

We proceed to analyze this operator ${\bf S}(\eta)$ more precisely. 
We put 
\[
U^m_b = \{ \eta \in H^m \,;\, 
 \inf_{{\boldsymbol x}\in\mathbf{R}^n}(h+\eta({\boldsymbol x})-b({\boldsymbol x}))>0 \},
\]
which is an open set in $H^m$.
For Banach spaces $\mathscr{X}$ and $\mathscr{Y}$, we denote by $B(\mathscr{X};\mathscr{Y})$ 
the set of all linear and bounded operators from $\mathscr{X}$ into $\mathscr{Y}$. 
In view of \eqref{IK:comp4}, \eqref{H:CV}, and Lemma~\ref{IK:lem1}, we see easily the following lemma.

\begin{lemma}\label{H:lem1}
Let $m$ be an integer such that $m>\frac{n}{2}+1$ and $b\in W^{m,\infty}$. 
For each $\eta \in U^m_b$ and for $k=1,2,\ldots,m$, the linear operator 
\[
{\bf S}(\eta) : \mathring{H}^k \ni\phi \mapsto \boldsymbol{\phi} \in\mathring{H}^k\times(H^k)^N
\]
is defined,
where $\boldsymbol{\phi}=(\phi_0,\phi_1,\ldots,\phi_N)^{\rm T}$ is the unique solution to \eqref{H:eqCV}. 
Moreover, we have ${\bf S}(\eta) \in B(\mathring{H}^k;\mathring{H}^k\times(H^k)^N)$ and 
\[
 L(H,b)\boldsymbol{\phi} = \bigl(L_0(H,b)\boldsymbol{\phi}\bigr)\boldsymbol{l}(H).
\]
\end{lemma}

Formally, $D_{\eta}{\bf S}(\eta)[\dot{\eta}]$ the Fr\'echet derivative of ${\bf S}(\eta)$ with respect to 
$\eta$ is given by 
\begin{equation}\label{H:Frechet}
\begin{cases}
 \boldsymbol{l}(H)\cdot\dot{\boldsymbol{\psi}}
  = -\bigl( \boldsymbol{l}'(H)\cdot\boldsymbol{\phi} \bigr)\dot{\eta}, \\
 \mathcal{L}(H,b)\dot{\boldsymbol{\psi}} = -D_H\mathcal{L}(H,b)[\dot{\eta}]\boldsymbol{\phi},
\end{cases}
\end{equation}
with $\boldsymbol{\phi}={\bf S}(\eta)\phi$ and $\dot{\boldsymbol{\psi}}=D_{\eta}{\bf S}(\eta)[\dot{\eta}]\phi$, 
where $\boldsymbol{l}'(H)\cdot\boldsymbol{\phi} =\sum_{j=1}^Np_jH^{p_j-1}\phi_j$, 
\[
D_H\mathcal{L}_i(H)[\dot{\eta}]\boldsymbol{\phi}
= \sum_{j=0}^N\bigl( D_HL_{ij}(H,b)[\dot{\eta}] - H^{p_i}D_HL_{0j}(H,b)[\dot{\eta}]
 - p_iH^{p_i-1}\dot{\eta}L_{0j}(H,b) \bigr)\phi_j,
\]
and 
\begin{align*}
D_H L_{ij}(H,b)[\dot{\eta}]\phi_j
&= -\nabla\cdot\bigl\{ \dot{\eta}( 
   H^{p_i+p_j}\nabla\phi_j
   -p_j H^{p_i+p_j-1}\phi_j\nabla b ) \bigr\} \nonumber \\[0.5ex]
&\quad\,
 + \dot{\eta}\bigl\{ -p_i H^{p_i+p_j-1}\nabla b\cdot\nabla\phi_j
   + p_ip_j H^{p_i+p_j-2}(1+|\nabla b|^2)\phi_j \bigr\}.
\end{align*}
By using these equations together with Lemma~\ref{IK:lem1} and standard arguments, 
we can justify the Fr\'echet differentiability of ${\bf S}(\eta)$ with respect to $\eta$. 
More precisely, we have the following lemma.

\begin{lemma}\label{H:lem2}
Let $m$ be an integer such that $m>\frac{n}{2}+1$ and $b\in W^{m,\infty}$. 
Then, the map $U^m_b \ni \eta \mapsto {\bf S}(\eta) \in B(\mathring{H}^k;\mathring{H}^k\times(H^k)^N)$ 
is Fr\'echet differentiable for $k=1,2,\ldots,m$, and~\eqref{H:Frechet} holds. 
\end{lemma}

As mentioned before, the Isobe--Kakinuma model~\eqref{intro:IK model} has a conserved quantity 
$E^{\mbox{\rm\scriptsize IK}}(\eta,\boldsymbol{\phi})$ given by~\eqref{IK:energy}, 
which is the total energy. 
Now, we define a Hamiltonian $\mathscr{H}^{\mbox{\rm\scriptsize IK}}(\eta,\phi)$ to the 
Isobe--Kakinuma model by 
\begin{equation}\label{H:H}
\mathscr{H}^{\mbox{\rm\scriptsize IK}}(\eta,\phi)
 = \frac{1}{\rho}E^{\mbox{\rm\scriptsize IK}}(\eta,{\bf S}(\eta)\phi),
\end{equation}
which is essentially the total energy in terms of the canonical variables $(\eta,\phi)$.

\begin{lemma}\label{H:lem3}
Let $m$ be an integer such that $m>\frac{n}{2}+1$ and $b\in W^{m,\infty}$. 
Then, the map 
$U^m_b \times\mathring{H}^1 \ni (\eta,\phi) \mapsto \mathscr{H}^{\mbox{\rm\scriptsize IK}}(\eta,\phi)\in\mathbf{R}$ 
is Fr\'echet differentiable and the variational derivatives of the Hamiltonian are 
\[
\begin{cases}
\delta_\phi\mathscr{H}^{\mbox{\rm\scriptsize IK}}(\eta,\phi) = L_0(H,b)\boldsymbol{\phi}, \\
\delta_\eta\mathscr{H}^{\mbox{\rm\scriptsize IK}}(\eta,\phi)
 = \frac{1}{\rho}(\delta_\eta E^{\mbox{\rm\scriptsize IK}})(\eta,\boldsymbol{\phi})
  -(\boldsymbol{l}'(H)\cdot\boldsymbol{\phi})L_0(H,b)\boldsymbol{\phi},
\end{cases}
\]
where $\boldsymbol{\phi}={\bf S}(\eta)\phi$. 
\end{lemma}
\begin{proof}[{\bf Proof}.]
Let us calculate Fr\'echet derivatives of the Hamiltonian $\mathscr{H}^{\mbox{\rm\scriptsize IK}}(\eta,\phi)$. 
Let us consider first 
$U^m_b \times H^2 \ni (\eta,\phi) \mapsto \mathscr{H}^{\mbox{\rm\scriptsize IK}}(\eta,\phi) $.
For any $\dot{\phi}\in {H}^2$, we see that 
\begin{align*}
D_\phi\mathscr{H}^{\mbox{\rm\scriptsize IK}}(\eta,\phi)[\dot{\phi}]
&= \frac{1}{\rho}(D_{\boldsymbol{\phi}} E^{\mbox{\rm\scriptsize IK}})
 (\eta,{\bf S}(\eta)\phi)[{\bf S}(\eta)\dot{\phi}] \\
&= \frac{1}{\rho}( (\delta_{\boldsymbol{\phi}} E^{\mbox{\rm\scriptsize IK}})(\eta,\boldsymbol{\phi}),
 {\bf S}(\eta)\dot{\phi})_{L^2} \\
&= (L(H,b)\boldsymbol{\phi},{\bf S}(\eta)\dot{\phi})_{L^2} \\
&= (\bigl(L_0(H,b)\boldsymbol{\phi}\bigr)\boldsymbol{l}(H), {\bf S}(\eta)\dot{\phi})_{L^2} \\
&= (L_0(H,b)\boldsymbol{\phi}, \boldsymbol{l}(H)\cdot {\bf S}(\eta)\dot{\phi})_{L^2} \\
&= ( L_0(H,b)\boldsymbol{\phi}, \dot{\phi})_{L^2},
\end{align*}
where we used \eqref{IK:deltaE} and Lemma~\ref{H:lem1}. 
The above calculations are also valid when $(\phi,\dot{\phi})\in  \mathring{H}^1\times \mathring{H}^1$, 
provided we replace the $L^2$ inner products with the $\mathscr{X}^\prime$--$\mathscr{X}$ duality product 
where $\mathscr{X} = \mathring{H}^1\times (H^1)^N$ for the first lines, and $\mathscr{X}= \mathring{H}^1$ 
for the last line. 
This gives the first equation of the lemma.

Similarly, for any $(\eta,\phi)\in U_b^m\times \mathring{H}^2$ and $\dot{\eta}\in H^m$ we see that 
\begin{align*}
D_\eta\mathscr{H}^{\mbox{\rm\scriptsize IK}}(\eta,\phi)[\dot{\eta}]
= \frac{1}{\rho}(D_\eta E^{\mbox{\rm\scriptsize IK}})(\eta,{\bf S}(\eta)\phi)[\dot{\eta}]
 + \frac{1}{\rho}(D_{\boldsymbol{\phi}} E^{\mbox{\rm\scriptsize IK}})(\eta,{\bf S}(\eta)\phi)
  [D_\eta {\bf S}(\eta)[\dot{\eta}]\phi].
\end{align*}
Here, we have 
\begin{align*}
\frac{1}{\rho}(D_{\boldsymbol{\phi}} E^{\mbox{\rm\scriptsize IK}})(\eta,{\bf S}(\eta)\phi)
  [D_\eta {\bf S}(\eta)[\dot{\eta}]\phi]
&= \frac{1}{\rho}( (\delta_{\boldsymbol{\phi}} E^{\mbox{\rm\scriptsize IK}})(\eta,\boldsymbol{\phi}),
  D_\eta {\bf S}(\eta)[\dot{\eta}]\phi)_{L^2} \\
&= ( L(H,b)\boldsymbol{\phi}, D_\eta {\bf S}(\eta)[\dot{\eta}]\phi)_{L^2} \\
&= ( L_0(H,b)\boldsymbol{\phi}, \boldsymbol{l}(H)\cdot D_\eta {\bf S}(\eta)[\dot{\eta}]\phi)_{L^2} \\
&= -( L_0(H,b)\boldsymbol{\phi}, (\boldsymbol{l}'(H)\cdot {\bf S}(\eta)\phi) \dot{\eta})_{L^2} \\
&= - ( (\boldsymbol{l}'(H)\cdot\boldsymbol{\phi})L_0(H,b)\boldsymbol{\phi}, \dot{\eta})_{L^2},
\end{align*}
where we used the identity
\[
\boldsymbol{l}(H)\cdot D_\eta {\bf S}(\eta)[\dot{\eta}]\phi
 + (\boldsymbol{l}'(H)\cdot {\bf S}(\eta)\phi) \dot{\eta} = 0,
\]
stemming from~\eqref{H:Frechet}. 
Again, the above identities are still valid for $(\eta,\phi)\in U_b^m\times \mathring{H}^1$ 
provided we replace the $L^2$ inner products with suitable duality products. 
This concludes the proof of the Fr\'echet differentiability, and the second equation of the lemma. 
\end{proof}

Now, we are ready to show our main result in this section.

\begin{theorem}
Let $m$ be an integer such that $m>\frac{n}{2}+1$ and $b\in W^{m,\infty}$. 
Then, the Isobe--Kakinuma model~\eqref{intro:IK model} is equivalent to Hamilton's canonical equations
\begin{equation}\label{H:CF}
\partial_t\eta = \frac{\delta\mathscr{H}^{\mbox{\rm\scriptsize IK}}}{\delta\phi}, \quad
\partial_t\phi = -\frac{\delta\mathscr{H}^{\mbox{\rm\scriptsize IK}}}{\delta\eta},
\end{equation}
with $\mathscr{H}^{\mbox{\rm\scriptsize IK}}$ defined in~\eqref{H:H} as long as $\eta(\cdot,t) \in U^m_b$ 
and $\phi(\cdot,t)\in \mathring{H}^1$. 
More precisely, for any regular solution $(\eta,\boldsymbol{\phi})$ to the Isobe--Kakinuma model 
\eqref{intro:IK model}, 
if we define $\phi$ by~\eqref{H:CV}, then $(\eta,\phi)$ satisfies Hamilton's canonical 
equations~\eqref{H:CF}. 
Conversely, for any regular solution $(\eta,\phi)$ to Hamilton's canonical equations~\eqref{H:CF}, 
if we define $\boldsymbol{\phi}$ by $\boldsymbol{\phi}={\bf S}(\eta)\phi$, 
then $(\eta,\boldsymbol{\phi})$ satisfies the Isobe--Kakinuma model~\eqref{intro:IK model}. 
\end{theorem}

\begin{proof}[{\bf Proof}.]
Suppose that $(\eta,\boldsymbol{\phi})$ is a solution to the Isobe--Kakinuma model~\eqref{intro:IK model}. 
Then, it satisfies~\eqref{IK:IK model2}, particularly, we have 
\begin{equation}\label{H:eq1}
\partial_t\eta = L_0(H,b)\boldsymbol{\phi}.
\end{equation}
It follows from~\eqref{H:CV} and~\eqref{IK:IK model2} that 
\begin{align*}
\partial_t\phi 
&= \boldsymbol{l}(H)\cdot\partial_t\boldsymbol{\phi}
 + (\boldsymbol{l}'(H)\cdot\boldsymbol{\phi})\partial_t\eta \\
&= -\frac{1}{\rho} (\delta_\eta E^{\mbox{\rm\scriptsize IK}})(\eta,\boldsymbol{\phi})
 + (\boldsymbol{l}'(H)\cdot\boldsymbol{\phi})L_0(H,b)\boldsymbol{\phi}.
\end{align*}
These equations together with Lemma~\ref{H:lem3} show that $(\eta,\phi)$ satisfies~\eqref{H:CF}.

Conversely, suppose that $(\eta,\phi)$ satisfies Hamilton's canonical equations~\eqref{H:CF} and 
put $\boldsymbol{\phi}={\bf S}(\eta)\phi$. 
Then, it follows from~\eqref{H:CF} and Lemma~\ref{H:lem3} that we have~\eqref{H:eq1}. 
This fact and Lemma~\ref{H:lem1} imply the equation 
\[
\boldsymbol{l}(H)\partial_t\eta = L(H,b)\boldsymbol{\phi}
 = \frac{1}{\rho}\delta_{\boldsymbol{\phi}}E^{\mbox{\rm\scriptsize IK}}(\eta,\boldsymbol{\phi}).
\]
We see also that 
\begin{align*}
\boldsymbol{l}(H)\cdot\partial_t\boldsymbol{\phi}
&= \partial_t\phi - (\boldsymbol{l}'(H)\cdot\boldsymbol{\phi})\partial_t\eta
 = - \frac{1}{\rho}\delta_{\eta}E^{\mbox{\rm\scriptsize IK}}(\eta,\boldsymbol{\phi}),
\end{align*}
where we used~\eqref{H:CF} and Lemma~\ref{H:lem3}. 
Therefore, $(\eta,\boldsymbol{\phi})$ satisfies~\eqref{IK:IK model2}, that is, 
the Isobe--Kakinuma model~\eqref{intro:IK model}. 
\end{proof}

%------------------------------------------------------------------------------
\section{Consistency}
\label{section:C}
As aforementioned, it was shown in~\cite{Iguchi2018-1, Iguchi2018-2} that the Isobe--Kakinuma model 
\eqref{intro:IK model} is a higher order shallow water approximation for the water wave problem 
in the strongly nonlinear regime. 
In this section, we will show that the canonical Hamiltonian structure exhibited 
in the previous section is consistent with this approximation, 
in the sense that the Hamiltonian of the Isobe--Kakinuma model, 
$\mathscr{H}^{\mbox{\rm\scriptsize IK}}(\eta,\phi)$, 
approximates the Hamiltonian of the water wave problem, 
$\mathscr{H}(\eta,\phi)$, in the shallow water regime.

In order to provide quantitative results, we first rewrite the equations in a nondimensional form. 
Let $\lambda$ be the typical wavelength of the wave.
Recalling that $h$ is the mean depth, we introduce the nondimensional aspect ratio 
\[
\delta = \frac{h}{\lambda} ,
\]
measuring the shallowness of the water. 
We then rescale the physical variables by 
\[
{\boldsymbol x} = \lambda\tilde {\boldsymbol x}, \quad z = h\tilde z, \quad 
t = \frac{\lambda}{\sqrt{gh}}\tilde t, \quad \eta = h\tilde\eta, \quad 
b= h\tilde b, \quad \Phi=\lambda\sqrt{gh}\tilde\Phi.
\]
Under these rescaling, after dropping the tildes for the sake of readability, 
the basic equations for water waves~\eqref{intro:Laplace}--\eqref{intro:BC2} 
are rewritten in a non-dimensional form 
\begin{equation}\label{C:Laplace}
 \Delta\Phi+\delta^{-2}\partial_z^2\Phi=0 \makebox[3em]{in}  \Omega(t), \; t>0, 
\end{equation}
\begin{equation}\label{C:BC1}
\begin{cases}
 \partial_t\eta + \nabla\Phi\cdot\nabla\eta - \delta^{-2}\partial_z\Phi = 0
  & \mbox{on}\quad  \Gamma(t), \; t>0,  \\
 \displaystyle
 \partial_t\Phi + \frac12\bigl( |\nabla\Phi|^2 +  \delta^{-2}(\partial_z\Phi)^2  \bigr) + \eta = 0
  & \mbox{on}\quad  \Gamma(t), \; t>0, 
\end{cases}
\end{equation}
\begin{equation}\label{C:BC2}
\nabla\Phi\cdot\nabla b - \delta^{-2} \partial_z\Phi = 0 \makebox[3em]{on}  \Sigma, \; t>0, 
\end{equation}
denoting $\Omega(t)$, $\Gamma(t)$, and $\Sigma$ the rescaled water region, water surface, and 
bottom of the water at time $t$, respectively. 
Specifically, the rescaled water surface and bottom of the water are represented as 
$z=\eta(\boldsymbol{x},t)$ and $z=-1+b(\boldsymbol{x})$, respectively. 
The corresponding dimensionless Zakharov--Craig--Sulem formulation is 
\begin{equation}\label{C:ZCS}
\begin{cases}
 \partial_t\eta-\Lambda^\delta(\eta,b)\phi = 0  & \mbox{on}\quad \mathbf{R}^n, \; t>0, \\[.5ex]
 \partial_t\phi + \eta + \dfrac12|\nabla\phi|^2
  - \dfrac{\delta^2}2\dfrac{\bigl(\Lambda^\delta(\eta,b)\phi + \nabla\eta\cdot\nabla\phi\bigr)^2
  }{1+\delta^2|\nabla\eta|^2} = 0 & \mbox{on}\quad \mathbf{R}^n, \; t>0,
\end{cases}
\end{equation}
where 
\begin{equation}\label{C:CV}
\phi({\boldsymbol x},t) = \Phi({\boldsymbol x},\eta({\boldsymbol x},t),t)
\end{equation}
is the trace of the velocity potential on the water surface, and 
$\Lambda^\delta(\eta,b)$ is the dimensionless Dirichlet-to-Neumann map for Laplace's equation, 
namely, it is defined by
\[
\Lambda^\delta(\eta,b)\phi = \delta^{-2} (\partial_z\Phi)|_{z=\eta} - \nabla\eta\cdot(\nabla\Phi)|_{z=\eta}, 
\]
where $\Phi$ is the unique  solution to the boundary value problem 
of the scaled Laplace's equation~\eqref{C:Laplace} under the boundary conditions 
\eqref{C:BC2} and \eqref{C:CV}. 
With this rescaling and definitions, the Hamiltonian of the water wave system is given by 
\[
\mathscr{H}^\delta(\eta,\phi)
= \frac12 \iint_{\Omega(t)} \bigl( |\nabla\Phi|^2 + \delta^{-2}(\partial_z\Phi)^2 \bigr)
  \,{\rm d}{\boldsymbol x}{\rm d}z + \frac{1}{2}\int_{\mathbf{R}^n}\eta^2 \,{\rm d}{\boldsymbol x} .
\]

In order to rewrite the Isobe--Kakinuma model~\eqref{intro:IK model} in dimensionless form, 
we need to rescale the unknown variables $(\phi_0,\phi_1,\ldots,\phi_N)$, 
depending on the choice of function system $\{\Psi_i\}$. 
In view of~\eqref{intro:app}, we rescale them by
\[
\phi_i = \frac{\lambda\sqrt{gh}}{\lambda^{p_i}} \tilde \phi_i \qquad\mbox{for}\quad i=0,1,\ldots,N, 
\]
so that 
\begin{equation}\label{C:app}
\Phi^{\rm app}({\boldsymbol x}, z, t)
= \lambda\sqrt{gh}\;\tilde\Phi^{\rm app}(\tilde{\boldsymbol x},\tilde z,\tilde t) 
= \lambda\sqrt{gh}\;\biggl(\sum_{i=0}^N\delta^{p_i}(\tilde z+1-\tilde b(\tilde{\boldsymbol x}))^{p_i} 
 \phi_i(\tilde{\boldsymbol x},\tilde t)\biggr).
\end{equation}
As before, we will henceforth drop the tildes for the sake of readability. 
It is also convenient to introduce the notation 
\[
\phi_i^\delta=\delta^{p_i}\tilde\phi_i \qquad\mbox{for}\quad i=0,1,\ldots,N,
\]
so that the Isobe--Kakinuma model~\eqref{intro:IK model} in rescaled variables is 
\begin{equation}\label{C:IK model}
\left\{
 \begin{array}{l}
  \displaystyle
  H^{p_i} \partial_t \eta + \sum_{j=0}^N \left\{ \nabla \cdot \left(
   \frac{1}{p_i+p_j+1} H^{p_i+p_j+1} \nabla\phi_j^\delta
   - \frac{p_j}{p_i+p_j} H^{p_i+p_j} \phi_j^\delta \nabla b \right) \right.\\
  \displaystyle\phantom{ H^{p_i} \partial_t \eta + \sum_{j=0}^N \biggl\{ }
   + \left. \frac{p_i}{p_i+p_j} H^{p_i+p_j} \nabla b \cdot \nabla\phi_j^\delta
   - \frac{p_ip_j}{p_i+p_j-1} H^{p_i+p_j-1} (\delta^{-2} + |\nabla b|^2) \phi_j^\delta \right\} = 0 \\
  \makebox[28em]{}\mbox{for}\quad i=0,1,\ldots,N, \\[1ex]
  \displaystyle
  \sum_{j=0}^N H^{p_j} \partial_t \phi_j^\delta + \eta 
   + \frac12 \left\{ \left| \sum_{j=0}^N 
   ( H^{p_j}\nabla\phi_j^\delta - p_j H^{p_j-1}\phi_j^\delta\nabla b ) \right|^2 
   + \delta^{-2}\left( \sum_{j=0}^N p_j H^{p_j-1} \phi_j^\delta \right)^2 \right\} = 0,
 \end{array}
\right.
\end{equation}
where $H({\boldsymbol x},t) = 1 + \eta({\boldsymbol x},t) - b({\boldsymbol x})$. 
We also use the notations 
$\boldsymbol{\phi}^\delta=(\phi_0^\delta,\phi_1^\delta,\dots,\phi_N^\delta)^{\rm T}$ 
and $L^\delta=L^\delta(H,b)=(L_{ij}^\delta(H,b))_{0\leq i,j\leq N}$, where 
\begin{align}
L_{ij}^\delta\psi_j
&= -\nabla\cdot\biggl(
   \frac{1}{p_i+p_j+1}H^{p_i+p_j+1}\nabla\psi_j
   -\frac{p_j}{p_i+p_j}H^{p_i+p_j}\psi_j\nabla b\biggr) \nonumber \\[0.5ex]
&\quad\,
  -\frac{p_i}{p_i+p_j}H^{p_i+p_j}\nabla b\cdot\nabla\psi_j
   +\frac{p_ip_j}{p_i+p_j-1}H^{p_i+p_j-1}(\delta^{-2}+|\nabla b|^2)\psi_j.
  \label{C:L}
\end{align}
Then, \eqref{C:IK model} can be written in a compact form 
\begin{equation}\label{C:IK model2}
\begin{pmatrix}
 0 & -\boldsymbol{l}(H)^{\rm T} \\
 \boldsymbol{l}(H) & O
\end{pmatrix}
\partial_t
\begin{pmatrix}
 \eta \\
 \boldsymbol{\phi}^\delta
\end{pmatrix}
= 
\begin{pmatrix}
 \delta_{\eta} E^{\mbox{\rm\scriptsize IK},\delta} \\
 \delta_{\boldsymbol{\phi}^\delta} E^{\mbox{\rm\scriptsize IK},\delta}
\end{pmatrix},
\end{equation}
where 
\begin{align}
E^{\mbox{\rm\scriptsize IK},\delta}(\eta,\boldsymbol{\phi}^\delta)
&= \frac{1}{2}\int_{\mathbf{R}^n}\left\{
 \sum_{i,j=0}^N\left(
  \frac{1}{p_i+p_j+1}H^{p_i+p_j+1}\nabla\phi_i^\delta\cdot\nabla\phi_j^\delta
  -\frac{2p_i}{p_i+p_j}H^{p_i+p_j}\phi_i^\delta\nabla b\cdot\nabla\phi_j^\delta \right. \right. \nonumber \\
& \left. \phantom{\sum_{i,j=0}^N}\makebox[6em]{}
 + \frac{p_ip_j}{p_i+p_j-1}H^{p_i+p_j-1}(\delta^{-2}+|\nabla b|^2)\phi_i^\delta\phi_j^\delta\biggr)
 + \eta^2\right\}{\rm d}{\boldsymbol x}.
\end{align}
Then, we define the Hamiltonian 
\[
\mathscr{H}^{\mbox{\rm\scriptsize IK},\delta}(\eta,\phi)
= E^{\mbox{\rm\scriptsize IK},\delta}(\eta,\boldsymbol{\phi}^\delta), 
\]
where $\boldsymbol{\phi}^\delta$ is the solution to 
\begin{equation}\label{C:S}
\begin{cases}
 \boldsymbol{l}(H)\cdot\boldsymbol{\phi}^\delta = \phi, \\
 L^\delta(H,b)\boldsymbol{\phi}^\delta
  = \bigl(L_0^\delta(H,b)\boldsymbol{\phi}^\delta\bigr)\boldsymbol{l}(H).
\end{cases}
\end{equation}
Here, we used the notation $L_0^\delta=(L_{00}^\delta,\dots,L_{0N}^\delta)$.
We recall that $\boldsymbol{\phi}^\delta$ is uniquely determined by~\eqref{C:S} thanks to Lemma~\ref{H:lem1}.

To analyze the consistency of the Hamiltonian in the shallow water regime, 
we will further restrict ourselves to the following two cases: 
\begin{itemize}
\item[(H1)]
In the case of the flat bottom $b({\boldsymbol x})\equiv0$, $p_i=2i$ for $i=0,1,\ldots,N$.
\item[(H2)]
In the case with general bottom topographies, $p_i=i$ for $i=0,1,\ldots,N$.
\end{itemize}
We are now in position to state the consistency of the Hamiltonian of the Isobe--Kakinuma model 
with respect to Zakharov's Hamiltonian of the water wave problem in the shallow water regime.

\begin{theorem}\label{C:consistency-Hamiltonian}
Let $c_0,M$ be positive constants and $m>\frac{n}{2}+1$ an integer such that $m\geq 4(N+1)$ 
in the case {\rm (H1)} and $m\geq 4([\frac{N}{2}]+1)$ in the case {\rm (H2)}. 
There exists a positive constant $C$ such that if $\eta\in H^m$ and $b\in W^{m+1,\infty}$ satisfy 
\[
\begin{cases}
 \|\eta\|_m + \|b\|_{W^{m+1,\infty}} \leq M, \\
 c_0\leq H({\boldsymbol x}) = 1+\eta({\boldsymbol x})-b({\boldsymbol x})
  \quad\mbox{for}\quad {\boldsymbol x}\in\mathbf{R}^n,
\end{cases}
\]
then for any $\delta\in(0,1]$ and any $\phi\in \mathring{H}^{m}$, we have 
\[
|\mathscr{H}^\delta(\eta,\phi) -\mathscr{H}^{\mbox{\rm\scriptsize IK},\delta}(\eta,\phi) |
\leq
\begin{cases}
 C\|\nabla\phi\|_{4N+3}\|\nabla\phi\|_{0}\, \delta^{4N+2} & \text{ in the case {\rm (H1)}}, \\
 C\|\nabla \phi\|_{4[\frac{N}{2}]+3}\|\nabla\phi\|_{0}\, \delta^{4[\frac{N}{2}]+2}
   & \text{ in the case {\rm (H2)}}.
\end{cases}
\]
\end{theorem}

\begin{remark}
Theorem 2.4 in~\cite{Iguchi2018-2} in fact states the stronger result that the difference between 
exact solutions of the water wave problem obtained in \cite{Iguchi2009,Lannes2013} and the corresponding solutions 
of the Isobe--Kakinuma model is bounded with the same order of precision as above on the relevant timescale. 
\end{remark}

\begin{remark}
It is important to notice that the order of the approximation given in Theorem~\ref{C:consistency-Hamiltonian} 
is greater than what we could expect based on~\eqref{C:app}, and in particular greater than the one obtained 
when using the Boussinesq expansion in the flat bottom case (H1): 
\[
\phi_{\rm B}(\tilde t,\tilde{\boldsymbol x})
 = \tilde\Phi^{\rm app}_{\rm B}(\tilde{\boldsymbol x}, \eta(\tilde{\boldsymbol x}, \tilde t),\tilde t) 
 \quad \mbox{with} \quad
\tilde\Phi^{\rm app}_{\rm B}(\tilde{\boldsymbol x}, \tilde z, \tilde t)
 = \sum_{i=0}^N\delta^{2i}(\tilde z+1)^{2i} \frac{(-\Delta)^i\phi_0(\tilde{\boldsymbol x},\tilde t)}{(2i)!}
\]
where $\phi_0$ is the trace of the velocity potential at the bottom. 
Here we can only expect that the approximation is valid up to an error of order $O(\delta^{2N+2})$, 
which coincides with the precision of Theorem~\ref{C:consistency-Hamiltonian} only when $N=0$. 
When $N=0$, we recover that the Saint-Venant or shallow-water equations provide approximate solutions 
with precision $O(\delta^2)$; see~\cite{Iguchi2009, Lannes2013}. 
\end{remark}

\begin{proof}[{\bf Proof of Theorem \ref{C:consistency-Hamiltonian}}.]
We will modify slightly the strategy in~\cite{Iguchi2018-2}. 
We first notice that 
\begin{align*}
&\mathscr{H}^\delta(\eta,\phi)
= \frac12\iint_{\Omega}\bigl( |\nabla\Phi|^2 + \delta^{-2}(\partial_z\Phi)^2 \bigr){\rm d}\boldsymbol{x}{\rm d}z
 + \frac12\int_{\mathbf{R}^n}\eta^2{\rm d}\boldsymbol{x}, \\
&\mathscr{H}^{\mbox{\rm\scriptsize IK},\delta}(\eta,\phi)
= \frac12\iint_{\Omega}\bigl( |\nabla\Phi^{\rm app}|^2
 + \delta^{-2}(\partial_z\Phi^{\rm app})^2 \bigr){\rm d}\boldsymbol{x}{\rm d}z
 + \frac12\int_{\mathbf{R}^n}\eta^2{\rm d}\boldsymbol{x}, 
\end{align*}
where $\Phi$ is the unique solution to the boundary value problem of the scaled Laplace's 
equation~\eqref{C:Laplace} under the boundary conditions \eqref{C:BC2} and \eqref{C:CV}, 
and the approximate velocity potential $\Phi^{\rm app}$ is defined by 
\[
\Phi^{\rm app}({\boldsymbol x},z)
= \sum_{i=0}^{N} ( z+1-b(\boldsymbol{x}))^{p_i}\phi_i^\delta({\boldsymbol x}),
\]
where
${\boldsymbol{\phi}}^\delta=({\phi}_0^\delta,{\phi}_1^\delta,\dots,{\phi}_{N}^\delta)^{\rm T}$ 
is the solution to 
\[
\begin{cases}
\displaystyle
 \sum_{i=0}^{N} H^{p_i} {\phi}_{i}^\delta = \phi, \\[3ex]
\displaystyle
 \sum_{j=0}^{N} {L}_{ij}^\delta(H,b){{\phi}}_j^\delta
  = H^{p_i}\sum_{j=0}^{N} {L}_{0j}^\delta(H,b){{\phi}}_j^\delta
 \makebox[4em]{for} i=0,1,\dots,N.
\end{cases}
\]
We will denote with tildes, as in~\cite{Iguchi2018-2}, the functions obtained 
when replacing $N$ with $2N+2$. 
Hence, $\widetilde{\boldsymbol{\phi}}^\delta 
= (\widetilde{\phi}_0^\delta,\widetilde{\phi}_1^\delta,\dots,\widetilde{\phi}_{2N+2}^\delta)^{\rm T}$ 
is the solution to 
\[
\begin{cases}
\displaystyle
 \sum_{i=0}^{2N+2} H^{p_i} \widetilde{\phi}_{i}^\delta = \phi, \\[3ex]
\displaystyle
 \sum_{j=0}^{2N+2} L_{ij}^\delta(H,b)\widetilde{{\phi}}_j^\delta
  = H^{p_i}\sum_{j=0}^{2N+2} L_{0j}^\delta(H,b)\widetilde{{\phi}}_j^\delta
  \makebox[4em]{for} i=0,1,\dots,2N+2.
\end{cases}
\]
We also introduce, as in \cite{Iguchi2018-2}, a modified approximate velocity potential
$\widetilde{\Phi}^{\rm app}$ by 
\begin{equation}\label{C:def-Phiapp}
\widetilde{\Phi}^{\rm app}({\boldsymbol x},z)
= \sum_{i=0}^{2N+2} ( z+1-b(\boldsymbol{x}))^{p_i}\widetilde\phi_i^\delta({\boldsymbol x}),
\end{equation}
and set $\Phi^{\rm res} = \Phi-\widetilde{\Phi}^{\rm app}$ and 
$\boldsymbol{\varphi}^\delta = (\varphi_0^\delta,\varphi_1^\delta,\ldots,\varphi_N^\delta)^{\rm T}$ with 
$\varphi_j^\delta=\phi_j^\delta-\widetilde{\phi}_j^\delta$ for $j=0,1,\ldots,N$. 
Then, we decompose the difference $\mathscr{H}^\delta - \mathscr{H}^{\mbox{\rm\scriptsize IK},\delta}$ as 
\begin{align*}
&\mathscr{H}^\delta(\eta,\phi) - \mathscr{H}^{\mbox{\rm\scriptsize IK},\delta}(\eta,\phi) \\
&= \frac12\iint_{\Omega}\bigl\{ \bigl( |\nabla\Phi|^2 + \delta^{-2}(\partial_z\Phi)^2 \bigr)
 - \bigl( |\nabla\widetilde{\Phi}^{\rm app}|^2 + \delta^{-2}(\partial_z\widetilde{\Phi}^{\rm app})^2 \bigr)
 \bigr\}{\rm d}\boldsymbol{x}{\rm d}z \\
&\quad\;
 + \frac12\iint_{\Omega}\bigl\{ \bigl( |\nabla\widetilde{\Phi}^{\rm app}|^2
  + \delta^{-2}(\partial_z\widetilde{\Phi}^{\rm app})^2 \bigr)-\bigl( |\nabla{\Phi}^{\rm app}|^2
    + \delta^{-2}(\partial_z{\Phi}^{\rm app})^2 \bigr)\bigr\}{\rm d}\boldsymbol{x}{\rm d}z \\
&= I_1 + I_2. 
\end{align*}

We first evaluate $I_1$. 
It is easy to see that 
\begin{align}
|I_1| &\leq \frac12\bigl\{
 \|\nabla\Phi^{\rm res}\|_{L^2(\Omega)}\bigl(
  \|\nabla\Phi\|_{L^2(\Omega)} + \|\nabla\widetilde{\Phi}^{\rm app}\|_{L^2(\Omega)} \bigr) \nonumber \\
&\quad\;\label{C:I1}
 + \delta^{-2}\|\partial_z\Phi^{\rm res}\|_{L^2(\Omega)}\bigl(
  \|\partial_z\Phi\|_{L^2(\Omega)} + \|\partial_z\widetilde{\Phi}^{\rm app}\|_{L^2(\Omega)} \bigr) \bigr\}.
\end{align}
By using \cite[Lemma~8.1]{Iguchi2018-2} with $k=0$ as well as~\cite[Lemma~6.4]{Iguchi2018-2} with 
$(k,j)=(0,2N+1)$ in the case (H1) and~\cite[Lemma~6.9]{Iguchi2018-2} with $(k,j)=(0,2[\frac{N}{2}]+1)$ 
in the case (H2), we find 
\[
\|\nabla\Phi^{\rm res}\|_{L^2(\Omega)} + \delta^{-1}\|\partial_z\Phi^{\rm res}\|_{L^2(\Omega)} \leq 
\begin{cases}
 C\|\nabla\phi\|_{4N+3}\;\delta^{4N+3} & \text{ in the case (H1)}, \\
 C\|\nabla \phi\|_{4[\frac{N}{2}]+3}\;\delta^{4[\frac{N}{2}]+3}  & \text{ in the case (H2)},
\end{cases}
\]
provided that $m\geq 4(N+1)$ in the case (H1), and $m\geq 4([\frac{N}{2}]+1)$ in the case (H2). 
Here and in what follows, $C$ denotes a positive constant depending on $N$, $m$, $c_0$, and $M$, 
which changes from line to line. 
On the other hand, it follows from elliptic estimates given in \cite{Iguchi2009, Lannes2013} that 
\[
\|\nabla\Phi\|_{L^2(\Omega)} + \delta^{-1}\|\partial_z\Phi\|_{L^2(\Omega)} \leq C\|\nabla\phi\|_0.
\]
Moreover, by the definition \eqref{C:def-Phiapp} and using~\cite[Lemma~3.4]{Iguchi2018-2} with $k=0$, 
we see that 
\begin{align*}
\|\nabla\widetilde{\Phi}^{\rm app}\|_{L^2(\Omega)}
 + \delta^{-1}\|\partial_z\widetilde{\Phi}^{\rm app}\|_{L^2(\Omega)}
&\leq C\bigl( \|\nabla\widetilde{\phi}_0^\delta\|_0
 + \|(\widetilde{\phi}_1^\delta,\ldots,\widetilde{\phi}_{2N+2}^\delta)\|_1
 + \delta^{-1}\|(\widetilde{\phi}_1^\delta,\ldots,\widetilde{\phi}_{2N+2}^\delta)\|_0 \bigr) \\
&\leq C\bigl( \|\nabla\widetilde{\phi}_0^\delta\|_0
 + \delta^{-1}\|(1-\delta^2\Delta)^{\frac12}
 (\widetilde{\phi}_1^\delta,\ldots,\widetilde{\phi}_{2N+2}^\delta)\|_0 \bigr) \\
&\leq C\|\nabla\phi\|_0. 
\end{align*}
Plugging the above estimates into~\eqref{C:I1}, we obtain
\begin{equation}\label{C:I1.5}
|I_1| \leq
\begin{cases}
 C\|\nabla\phi\|_{4N+3} \|\nabla\phi\|_0 \, \delta^{4N+3} & \text{ in the case (H1)}, \\
 C\|\nabla \phi\|_{4[\frac{N}{2}]+3} \|\nabla\phi\|_0 \, \delta^{4[\frac{N}{2}]+3}  & \text{ in the case (H2)},
\end{cases}
\end{equation}
provided that $m\geq 4(N+1)$ in the case (H1), and $m\geq 4([\frac{N}{2}]+1)$ in the case (H2).

We proceed to evaluate $I_2$ by noticing that, after the calculations in~\cite[p.~2009]{Iguchi2018-2},
\begin{align*}
I_2 &= \frac12\sum_{i=0}^{2N+2}\sum_{j=0}^{2N+2}
 \bigl( L_{ij}^\delta(H,b)\tilde{\phi}_j^\delta,\tilde{\phi}_i^\delta \bigr)_{L^2}
 - \frac12\sum_{i=0}^N\sum_{j=0}^N
 \bigl( L_{ij}^\delta(H,b)\phi_j^\delta,\phi_i^\delta \bigr)_{L^2} \\
&= \frac12\sum_{j=0}^{N}\sum_{i=N+1}^{2N+2} \big( 
 (L_{ij}^\delta(H,b) - H^{p_i}L_{0j}^\delta(H,b))\varphi_j^\delta, \widetilde{\phi}_{i}^\delta \big)_{L^2} \\
&\quad\;
 - \frac12\sum_{j=N+1}^{2N+2}\sum_{i=N+1}^{2N+2} \big( 
  (L_{ij}^\delta(H,b) - H^{p_i}L_{0j}^\delta(H,b))\widetilde{{\phi}}_j^\delta, 
   \widetilde{\phi}_{i}^\delta \big)_{L^2}.
\end{align*}
Therefore, 
\begin{align*}
|I_2| &\leq C \bigl\{ 
 \|\boldsymbol{\varphi}^\delta\|_{2N^*+3}
  + \| (\widetilde{\phi}_{N+1}^\delta,\ldots,\widetilde{\phi}_{2N+2}^\delta) \|_{2N^*+3}  \\
&\quad\;
 + \delta^{-2} \bigl( \|\boldsymbol{\varphi}^\delta\|_{2N^*+1}
  + \| (\widetilde{\phi}_{N+1}^\delta,\ldots,\widetilde{\phi}_{2N+2}^\delta) \|_{2N^*+1}\bigr) \bigr\}
 \| (\widetilde{\phi}_{N+1}^\delta,\ldots,\widetilde{\phi}_{2N+2}^\delta) \|_{-(2N^*+1)}
\end{align*}
with $N^\star=N$ in the case (H1) and $N^\star=[\frac{N}{2}]$ in the case (H2). 
Using~\cite[Lemma~6.2]{Iguchi2018-2} with $(k,j) = (2N+3,N),(2N+1,N+1),(-2N-1,N+1)$ in the case (H1) 
and~\cite[Lemma~6.7]{Iguchi2018-2} with $(k,j) = (2[\frac{N}{2}]+3,[\frac{N}{2}]),
(2[\frac{N}{2}]+1,[\frac{N}{2}]+1),(-2[\frac{N}{2}]-1,[\frac{N}{2}]+1)$ in the case (H2), we obtain 
\begin{equation}\label{C:I2}
|I_2| \leq 
\begin{cases}
 C\|\nabla\phi\|_{4N+2}\|\nabla\phi\|_0\;\delta^{4N+2} & \text{ in the case (H1)}, \\
  C\|\nabla\phi\|_{4[\frac{N}{2}]+2}\|\nabla\phi\|_0\;\delta^{4[\frac{N}{2}]+2}  & \text{ in the case (H2)},
\end{cases}
\end{equation}
provided that $m\geq 4N+3$ in the case (H1), and $m\geq 4[\frac{N}{2}]+3$ in the case (H2). 
Now, \eqref{C:I1.5} and \eqref{C:I2} give the desired estimate. 
\end{proof}

%------------------------------------------------------------------------------

\bigskip
Vincent Duch\^ene \par
{\sc Institut de Recherche Math\'ematique de Rennes} \par
{\sc Univ Rennes}, CNRS, IRMAR -- UMR 6625 \par
{\sc F-35000 Rennes, France} \par
E-mail: \texttt{vincent.duchene@univ-rennes1.fr}

\bigskip
Tatsuo Iguchi \par
{\sc Department of Mathematics} \par
{\sc Faculty of Science and Technology, Keio University} \par
{\sc 3-14-1 Hiyoshi, Kohoku-ku, Yokohama, 223-8522, Japan} \par
E-mail: \texttt{iguchi@math.keio.ac.jp}

\end{document}